\documentclass{amsart}

\usepackage{amssymb}

\usepackage{epsfig}

\newcommand{\aut}[0]{\operatorname{Aut}}  

\newcommand{\Aut}[0]{\operatorname{Aut}}   
\newcommand{\Bir}[0]{\operatorname{Bir}}

\newcommand{\diff}[0]{\operatorname{Diff}}  

\newcommand{\Pic}[0]{\operatorname{Pic}}

\newcommand{\PGL}[0]{\operatorname{PGL}}
\newcommand{\Gal}[0]{\operatorname{Gal}}

\newcommand{\PO}[0]{\operatorname{PO}}

\newcommand{\sS}[0]{{\mathbb S}}
\newcommand{\p}[0]{{\mathbb P}}
\newcommand{\F}[0]{{\mathbb F}}

\renewcommand{\c}[0]{{\mathbb C}}
\newcommand{\C}{{\mathbb C}}

\renewcommand{\r}[0]{{\mathbb R}} 
\newcommand{\s}[0]{{\mathcal S}} 
\renewcommand{\d}[0]{{\mathcal D}} 
\newcommand{\I}{\mathrm{I}}
\newcommand{\II}{\mathrm{II}}
\newcommand{\III}{\mathrm{III}}
\newcommand{\IV}{\mathrm{IV}}

\newcommand{\im}{{\mathbf i}}

\renewcommand{\epsilon}[0]{\varepsilon}

\numberwithin{equation}{section}%

\newtheorem{thm}[equation]{Theorem}%
\newtheorem{prop}[equation]{Proposition}
\newtheorem{lem}[equation]{Lemma}
\newtheorem{cor}[equation]{Corollary}

\theoremstyle{remark}
\newtheorem{rem}[equation]{Remark}

\newtheorem{example}[equation]{Example}

\theoremstyle{definition}

\newtheorem{dfn}[equation]{Definition}
\newtheorem{defn}[equation]{Definition}

  \usepackage[all]{xy}
\xyoption{matrix}

\title[Cremona groups of real surfaces]{Cremona groups of real surfaces}
\thanks{First author supported by the SNSF grant no PP00P2\_128422 /1}
\thanks{This research was partially supported by ANR Grant "BirPol"  ANR-11-JS01-004-01}
\author{J\'er\'emy Blanc}
\address{J\'er\'emy Blanc,
Mathematisches Institut,
Universit\"at Basel,
Rheinsprung 21,
CH-4051 Basel,
Schweiz}
\email{Jeremy.Blanc@unibas.ch}

\author{Fr\'ed\'eric Mangolte}
\address{Fr\'ed\'eric Mangolte,
   LUNAM Universit\'e, LAREMA, Universit\'e d'Angers, Bd. Lavoisier, 49045 Angers Cedex 01, France} 
   \email{frederic.mangolte@univ-angers.fr}

\begin{document}

\begin{abstract}
We give an explicit set of generators for various natural subgroups of the real Cremona group $\Bir_\r(\p^2)$. This completes and unifies former results by several authors.
\end{abstract}

\maketitle

\begin{quote}\small
\textit{MSC 2000:} 14E07, 14P25, 14J26
\par\medskip\noindent
\textit{Keywords:} real algebraic surface, rational surface, birational geometry,
algebraic automorphism, Cremona transformation
\end{quote}

\section{Introduction}
\subsection{On the real Cremona group $\Bir_\r(\p^2)$}
The classical Noether-Castelnuovo Theorem (1917) gives generators of the group $\Bir_\c(\p^2)$ of birational transformations of the complex projective plane. The group is generated by the  biregular automorphisms, which form the group $\Aut_\c(\p^2)\cong \PGL(3,\C)$ of projectivities, and by the standard quadratic transformation 
$$
\sigma_0\colon (x:y:z)\dasharrow (yz:xz:xy).
$$

This result does not work over the real numbers. Indeed, recall that a \emph{base point} of a birational transformation is a (possibly infinitely near) point of indeterminacy; and note that the base points of the quadratic involution
$$
\sigma_1\colon (x:y:z)\dasharrow (y^2+z^2: xy: xz)
$$
are not real. Thus $\sigma_1$ cannot be generated by projectivities and $\sigma_0$. More generally, we cannot generate this way maps having non real  base-points.
Hence the group $\Bir_\r(\p^2)$ of birational transformations of the real projective plane is not generated by $\Aut_\r(\p^2)\cong \PGL(3,\r)$  and $\sigma_0$.

The first result of this note is that $\Bir_\r(\p^2)$ is generated by $\Aut_\r(\p^2)$, $\sigma_0$, $\sigma_1$, and a family of birational maps of degree $5$ having only non real base-points.

\begin{thm}\label{thm:BirP2}
The group $\Bir_\r(\p^2)$ is generated by $\Aut_\r(\p^2)$, $\sigma_0$, $\sigma_1$, and the \emph{standard quintic} transformations of $\p^2$ $($defined in Example~$\ref{Exa:QuinticStd})$.
\end{thm}

The proof of this result follows the so-called Sarkisov program, which amounts to decompose a birational map between Mori fiber spaces as a sequence of simple maps, called \emph{Sarkisov links}. The description of all possible links has been done in \cite{IskFact} for perfect fields, and in \cite{bib:Pol} for real surfaces. We recall it in Section~\ref{Sec:Mori} and show how to deduce Theorem~\ref{thm:BirP2} from the list of Sarkisov links.

Let $X$ be an algebraic variety defined over $\r$, we denote as usual by $X(\r)$ the set of real points endowed with the induced algebraic structure. The topological space $\p^2(\r)$ is then the real projective plane, letting $\F_0:=\p^1\times\p^1$, the space $\F_0(\r)$ is the torus $\sS^1\times\sS^1$  and letting $Q_{3,1}=\{(w:x:y:z)\in \p^3\ |\ w^2=x^2+y^2+z^2\}$, the real locus $Q_{3,1}(\r)$ is the sphere~$\sS^2$.

Recall that an \emph{automorphism} of $X(\r)$ is a birational transformation $\varphi\in\Bir_\r(X)$ such that $\varphi$ and $\varphi^{-1}$ are defined at all real points of $X$. The set of such maps form a group $\Aut(X(\r))$, and we have natural inclusions
 $$\Aut_\r(X)\subset \Aut(X(\r))\subset\Bir_\r(X).$$

The strategy used to prove Theorem~\ref{thm:BirP2} allows us to treat similarly the case of natural subgroups of $\Bir_\r(\p^2)$, namely the groups $\Aut(\p^2(\r))$, $\Aut(Q_{3,1}(\r))$ and $\Aut(\F_0(\r))$ of the three \emph{minimal} real rational surfaces (see \ref{cor.mini}). This way, we give a unified treatment to prove three theorems on generators, the first two of them already proved in a different way in \cite{rv} and \cite{km1}. 

Observe that $\Aut(Q_{3,1}(\r))$ and $\Aut(\F_0(\r))$ are not really subgroups of $\Bir_\r(\p^2)$, but each of them is isomorphic to a subgroup which is determined up to conjugation. Indeed, for any choice of a birational map $\psi\colon \p^2 \dashrightarrow X$ ($X=Q_{3,1}$ or $\F_0$), $\psi^{-1}\Aut(X(\r))\psi\subset \Bir_\r(\p^2)$.


\begin{thm}[\cite{rv}]\label{ThmRongaVust}
The group $\Aut(\p^2(\r))$ is generated by $\Aut_\r(\p^2)=\PGL(3,\r)$ and by standard quintic transformations.
\end{thm}
\begin{thm}[\cite{km1}]\label{Thmkm}
The group $\Aut(Q_{3,1}(\r))$ is generated by $\Aut_\r(Q_{3,1})=\PO(3,1)$  and by standard cubic transformations.
\end{thm}
\begin{thm}\label{ThmTorus}
The group $\Aut(\F_0(\r))$ is generated by $\Aut_\r(\F_0)=\PGL(2,\r)^2\rtimes \mathbb{Z}/2\mathbb{Z}$  and by the involution
$$
\tau_0\colon ((x_0:x_1),(y_0:y_1))\dasharrow ((x_0:x_1),(x_0y_0+x_1y_1:x_1y_0-x_0y_1)).
$$
\end{thm}
The proof of theorems~\ref{thm:BirP2}, \ref{ThmRongaVust}, \ref{Thmkm}, \ref{ThmTorus} is given in Sections~\ref{BirP2R},~\ref{AutP2R}, \ref{Q31}, \ref{F0}, respectively. Section~\ref{Other} is devoted to present some related recent results on birational geometry of real projective surfaces.

\medskip

In the sequel, surfaces and maps are assumed to be real. In particular if we consider that a real surface is a complex surface endowed with a Galois-action of $G:=\Gal(\c\vert \r)$, a map is $G$-equivariant. On the contrary, points and curves are not assumed to be real a priori.

\section{Mori theory for real rational surfaces and Sarkisov program}\label{Sec:Mori}

We work with the tools of Mori theory. A good reference in dimension $2$, over any perfect field, is \cite{IskFact}.
The theory, applied to smooth projective real rational surfaces, becomes really simple. The description of Sarkisov links between real rational surfaces has been done in \cite{bib:Pol}, together with a study of relations between these links. In order to state this classification, we first recall  the following classical definitions (which can be found in \cite{IskFact}).

 \begin{dfn}
 A smooth projective real rational surface $X$ is said to be \emph{minimal} if any birational morphism $X\to Y$, where $Y$ is another smooth projective real surface, is an isomorphism.\end{dfn}
 
 \begin{dfn}
A \emph{Mori fibration} is a morphism $\pi\colon X\to W$ where $X$ is a smooth projective real rational surface and one of the following occurs
 
 \begin{enumerate}
 \item
 $\rho(X)=1$, $W$ is a point (usually denoted $\{*\}$), and $X$ is a del Pezzo surface;
 \item
 $\rho(X)=2$, $W= \p^1$  and the map $\pi$ is a conic bundle.
 \end{enumerate}
 \end{dfn}
Note that for an arbitrary surface, the curve $W$ in the second case should be any smooth curve, but we restrict ourselves to rational surfaces which implies that $W$ is isomorphic to $\p^1$.
 
 \begin{prop}
 Let $X$ be a smooth projective real rational surface. If $X$ is minimal, then it admits a morphism $\pi\colon X\to W$ which is a Mori fibration.
 \end{prop}
 \begin{proof}
 Follows from \cite{bib:IskMinimal}. See also \cite{bib:Mor}.
 \end{proof}
 
 \begin{defn}
 A \emph{Sarkisov link} between two Mori fibrations $\pi_1\colon X_1\to W_1$ and $\pi_2\colon X_2\to W_2$ is a birational map $\varphi\colon X_1\dasharrow X_2$ of one of the following four types, where each of the diagrams is commutative:
 
 \begin{enumerate}
 \item {\sc Link of Type $\I$} 
$$\xymatrix{
X_1\ar[d]_{\pi_1}\ar@{-->}[r]_\varphi& X_2\ar[d]_{\pi_2}\\
\{*\}=W_1& \ar[l]_{\tau} W_2=\p^1
}$$
where $\varphi^{-1}\colon X_2\to X_1$ is a birational morphism, which is the blow-up of either a real point or two imaginary conjugate points of $X_1$, and where $\tau$ is the contraction of $W_2=\p^1$ to the point $W_1$.

 \item {\sc Link of Type $\II$}
$$\xymatrix{
X_1\ar@/_0.8pc/@{-->}[rr]_{\varphi}\ar[d]_{\pi_1}& Z\ar[l]_{\sigma_1}\ar[r]^{\sigma_2}& X_2\ar[d]_{\pi_2}\\
W_1\ar[rr]_{\tau}^{\simeq}&& W_2
}$$
where $\sigma_i\colon Z\to X_i$ is a birational morphism, which is the blow-up of either a real point  or two imaginary conjugate points of $X_i$, and where $\tau$ is an isomorphism between $W_1$ and $W_2$.
 \item {\sc Link of Type $\III$} 
$$\xymatrix{
X_1\ar[r]_{\varphi}\ar[d]_{\pi_1}& X_2\ar[d]_{\pi_2}\\
\p^1=W_1\ar[r]_{\tau}&  W_2=\{*\}
}$$
where $\varphi\colon X_1\to X_2$ is a birational morphism, which is the blow-up of either a real point or two imaginary conjugate points of $X_2$, and where $\tau$ is the contraction of $W_1=\p^1$ to the point $W_2$. (\emph{It is the inverse of a link of type~$\mathrm{I}$}.)
 \item {\sc Link of Type $\IV$} 
$$\xymatrix{
X_1\ar[r]_{\varphi}^{\simeq}\ar[d]_{\pi_1}& X_2\ar[d]_{\pi_2}\\
\p^1=W_1&  W_2=\p^1
}$$
where $\varphi\colon X_1\to X_2$ is an isomorphism and $\pi_1,\pi_2\circ \varphi$ are conic bundles on $X_1$ with distinct fibres.
 \end{enumerate}
 \end{defn}
 Note that the morphism $\tau$ is important only for links of type $\mathrm{II}$, between two surfaces with a Picard group of rank $2$ (in higher dimension $\tau$ is important also for other links).\\
 
 \begin{defn}
 If $\pi\colon X\to W$ and $\pi'\colon X'\to W'$ are two (Mori) fibrations, an isomorphism $\psi\colon X\to X'$ is called an \emph{isomorphism of fibrations} if there exists an isomorphism $\tau\colon W\to W'$ such that $\pi'\psi=\tau\pi$.
 \end{defn}
 Note that the composition $\alpha\varphi\beta$ of a Sarkisov link $\varphi$ with some automorphisms of fibrations $\alpha$ and $\beta$ is again a Sarkisov link. 
  We have the following fundamental result:
 \begin{prop}\label{Prop:Sarkisov}
 If $\pi\colon X\to W$ and $\pi'\colon X'\to W'$ are two Mori fibrations, then any birational map $\psi\colon X\dasharrow X'$ is either an isomorphism of fibrations or decomposes into Sarkisov links.
 \end{prop}
 \begin{proof}Follows from \cite{IskFact}. See also \cite{bib:Corti}.\end{proof}
 
 \begin{thm}[\cite{Com1} (see also \cite{bib:IskMinimal})]\label{thm.mini}
Let $X$ be a real rational surface, if $X$ is minimal, then it is isomorphic to one of the following:
  \begin{enumerate}
  \item
   $\p^2$, 
  \item the quadric  $Q_{3,1}=\{(w:x:y:z)\in \p^3\ |\ w^2=x^2+y^2+z^2\}$, 
  \item a Hirzebruch surface $\mathbb{F}_n=\{((x:y:z),(u:v))\in \p^2 \times \p^1\ |\ yv^n=zu^n\}$ with $n\ne 1$.   \end{enumerate}
  \end{thm}
  
By \cite{Ma06}, if $n-n'\equiv 0 \mod 2$,  $\mathbb{F}_n(\r)$ is isomorphic to $\mathbb{F}_{n'}(\r)$, we get:
  
  \begin{cor}\label{cor.mini}
Let $X(\r)$ be the real locus of a real rational surface. If $X$ is minimal, then $X(\r)$ is  isomorphic to one of the following:
  \begin{enumerate}
  \item
   $\p^2(\r)$, 
  \item $Q_{3,1}(\r)\sim \sS^2$, 
  \item $\F_0(\r)\sim \sS^1\times\sS^1$.
  \end{enumerate}
  \end{cor}
 
We  give a list of Mori fibrations on real rational surfaces, and will show that, up to isomorphisms of fibrations, this list is exhaustive.
 \begin{example}\label{Exa:MoriFib}
 The following morphisms $\pi\colon X\to W$ are Mori fibrations on the plane, the sphere, the Hirzebruch surfaces, and a particular Del Pezzo surface of degree $6$.
 \begin{enumerate}
 \item
 $\p^2\to\{*\}$;
 \item
 $Q_{3,1}=\{(w:x:y:z)\in \p^3_\r\ |\ w^2=x^2+y^2+z^2\}\to\{*\}$;
 \item
 $\mathbb{F}_n=\{((x:y:z),(u:v))\in \p^2 \times \p^1\ |\ yv^n=zu^n\}\to \p^1$ for $n\ge 0$  (the map is the projection on the second factor and $\mathbb{F}_n$ is the $n$-th Hirzebruch surface);
 \item
 $\d_6=\{(w:x:y:z),(u:v)\in Q_{3,1}\times \p^1\ |\ wv=xu\}\to \p^1$ (the map is the projection on the second factor).
 \end{enumerate}
 \end{example}
 
 \begin{example}\label{Exa:Sarkisov}
 The following maps between the surfaces of Example~\ref{Exa:MoriFib} are Sarkisov links:
 \begin{enumerate}
 \item
The contraction of the exceptional curve of $\mathbb{F}_1$ (or equivalently the blow-up of a real point of $\p^2$),  is a link $\mathbb{F}_1\to \p^2$  of type $\III$. Note that the converse of this link is of type $\I$.
 \item
 The stereographic projection from the North pole $p_N=(1:0:0:1)$, $\varphi\colon Q_{3,1}\dasharrow \p^2$ given by $\varphi\colon (w:x:y:z)\dasharrow (x:y:w-z)$ and its inverse $\varphi^{-1}\colon \p^2\dasharrow Q_{3,1}$ given by $\varphi^{-1}\colon (x:y:z)\dasharrow (x^2+y^2+z^2:2xz:2yz:x^2+y^2-z^2)$ are both Sarkisov links of type $\II$.
 
 The map $\varphi$ decomposes into the blow-up of $p_N$, followed by the contraction of the strict transform of the curve $z=w$ (intersection of $Q_{3,1}$ with the tangent plane at $p_N$), which is the union of two imaginary conjugate lines. The map $\varphi^{-1}$ decomposes into the blow-up of the two imaginary points $(1:\pm \im :0)$, followed by the contraction of the strict transform of the line $z=0$.
 
 \item
The projection on the first factor $\d_6\to Q_{3,1}$ which contracts the two disjoint conjugate imaginary $(-1)$-curves $(0:0:1:\pm \im)\times \p^1\subset \d_6$ onto the two conjugate imaginary points $(0:0:1:\pm \im)\in Q_{3,1}$ is a link of type $\III$.
 \item
 The blow-up of a real point  $q\in \mathbb{F}_n$, lying on the exceptional section if $n>0$ (or any point if $n=0$), followed by the contraction of the strict transform of the fibre passing through $q$ onto a real point of $\mathbb{F}_{n+1}$ not lying on the exceptional section  is  a link $\mathbb{F}_n\dasharrow \mathbb{F}_{n+1}$ of type $\II$.
 \item
The blow-up of two conjugate imaginary points $p,\bar{p}\in \mathbb{F}_n$ lying on the exceptional section if $n>0$, or on the same section of self-intersection $0$ if $n=0$, followed by the contraction of the strict transform of the fibres passing through $p,\bar{p}$ onto two imaginary conjugate points of $\mathbb{F}_{n+2}$ not lying on the exceptional section  is a link $\mathbb{F}_n\dasharrow \mathbb{F}_{n+2}$ of type $\II$.
 \item
The blow-up of two conjugate imaginary points $p,\bar{p}\in \mathbb{F}_n$, $n\in\{0,1\}$ not lying on the same fibre  (or equivalently not lying on a real fibre) and  not on the same section of self-intersection $-n$ (or equivalently not lying on a real section of self-intersection $-n$), followed by the contraction 
 of the fibres passing through $p,\bar{p}$ onto two imaginary conjugate points of $\mathbb{F}_{n}$ having the same properties  is a link $\mathbb{F}_n\dasharrow \mathbb{F}_{n}$ of type $\II$.
 \item
The exchange of the two components $\p^1\times \p^1\to \p^1\times\p^1$ is a link $\mathbb{F}_0\to \mathbb{F}_0$ of type $\IV$.
\item
The blow-up of a real point $p\in\d_6$, not lying on a singular fibre (or equivalently $p\not=\left((1:1:0:0),(1:1)\right)$, $p\not=\left((1:-1:0:0),(1:-1)\right)$), followed by the contraction of the strict transform of the  fibre passing through $p$ onto a real point of $\d_6$,  is a link $\d_6\dasharrow \d_6$ of type $\II$. 
\item
The blow-up of two imaginary conjugate points $p,\bar{p}\in \d_6$, not lying on the same fibre (or equivalently not lying on a real fibre), followed by the contraction of the strict transform of the  fibres passing through $p,\bar{p}$ onto two imaginary points of $\d_6$ is a link $\d_6\dasharrow \d_6$ of type $\II$.  
 \end{enumerate}
 \end{example}
 \begin{rem}
Note that in the above list, the three links  $\mathbb{F}_n\dasharrow \mathbb{F}_m$ of type $\II$ can be put in one family,  and the same is true for the two links $\d_6\dasharrow \d_6$. We distinguished here the possibilities for the base points to describe more precisely the geometry of each link. The two links $\d_6\dasharrow \d_6$ could also be arranged into extra families, by looking if the base points belong to the two exceptional sections of self-intersection $-1$, but go in any case from $\d_6$ to $\d_6$.\end{rem}
 
 \begin{prop}
 Any Mori fibration $\pi \colon X\to W$, where $X$ is a smooth projective real rational surface, belongs to the list of Example~$\ref{Exa:MoriFib}$.
 
 Any Sarkisov link between two such Mori fibrations is equal to $\alpha\varphi\beta$, where $\varphi$ or $\varphi^{-1}$ belongs to the list described in Example~$\ref{Exa:Sarkisov}$ and where $\alpha$ and $\beta$ are isomorphisms of fibrations. 
 \end{prop}
 \begin{proof}
Since any birational map between two surfaces with Mori fibrations decomposes into Sarkisov links and that all links of Example~$\ref{Exa:Sarkisov}$ involve only the Mori fibrations of Example~$\ref{Exa:MoriFib}$, it suffices to check that any link starting from one of the Mori fibrations of~$\ref{Exa:MoriFib}$ belongs to the list~$\ref{Exa:Sarkisov}$. This is an easy case-by-case study; here are the steps.
 
Starting from a Mori fibration $\pi \colon X\to W$ where $W$ is a point, the only links we can perform are links of type $\mathrm{I}$ or $\mathrm{II}$ centered at a real point or two conjugate imaginary points. From \ref{thm.mini}, the surface $X$ is either $Q_{3,1}$ or $\p^2$, and both are homogeneous under the action of $\Aut(X)$, so the choice of the point is not relevant. Blowing-up a real point in $\p^2$  or two imaginary points in $Q_{3,1}$ gives rise to a link of type~$\mathrm{I}$ to $\mathbb{F}_1$ or $\d_6$. The remaining cases correspond to the stereographic projection $Q_{3,1}\dasharrow \p^2$ and its converse.

 Starting from a Mori fibration $\pi \colon X\to W$ where $W=\p^1$, we have additional possibilities. If the link is of type $\mathrm{IV}$, then $X$ admits two conic bundle structures and by \ref{thm.mini}, the only possibility is $\mathbb{F}_0=\p^1 \times \p^1$. If the link is of type $\mathrm{III}$, then we contract a real $(-1)$-curve of $X$ or two disjoint conjugate imaginary $(-1)$-curves. The only possibilities for $X$ are respectively $\mathbb{F}_1$ and $\d_6$, and the image is respectively $\p^2$ and $Q_{3,1}$ (these are the inverses of the links described before). The last possibility is to perform a link a type $\mathrm{II}$, by blowing up a real point or two conjugate imaginary points, on respectively one or two smooth fibres, and to contract the strict transform. We go from $\d_6$ to $\d_6$ or from $\mathbb{F}_m$ to $\mathbb{F}_{m'}$ where $m'-m\in \{-2,-1,-0,1,1\}$. All possibilities are described in Example~\ref{Exa:Sarkisov}.
 \end{proof}
 
 We end this section by reducing the number of links of type $\mathrm{II}$ needed for the classification. For this, we introduce the notion of standard links.
  \begin{defn}
 The following links of type $\mathrm{II}$ are called \emph{standard}:
 
 \begin{enumerate}
 \item
 links $\F_m\dasharrow \F_n$, with $m,n\in \{0,1\}$;
 \item
 links $\d_6\dasharrow \d_6$ which do not blow-up any point on the two exceptional section of self-intersection $-1$.
 \end{enumerate}
 The other links of type $\mathrm{II}$ will be called \emph{special}.
 \end{defn}
 
 The following result allows us to simplify the set of generators of our groups.
 
 \begin{lem}\label{Lem:4is123}
 Any Sarkisov link of type $\mathrm{IV}$ decomposes into links of type $\mathrm{I}$, $\mathrm{III}$, and standard links of type $\mathrm{II}$.
 \end{lem}\begin{proof}
 Note  that a link of type $\mathrm{IV}$ is, up to automorphisms preserving the fibrations, equal to the following automorphism of $\p^1\times\p^1$
  $$\tau\colon ((x_1:x_2),(y_1:y_2))\mapsto ((y_1:y_2),(x_1:x_2)).$$
   We denote by $\psi\colon \p^2\dasharrow \p^1\times \p^1$ the birational map $(x:y:z)\dasharrow ((x:y),(x:z))$ and observe that $ \tau\psi=\psi\sigma$, where $\sigma\in \Aut_\r(\p^2)$. Hence, $\tau=\psi\tau\psi^{-1}$. Observing that $\psi$ decomposes into the blow-up of the point $(0:0:1)$, which is a link of type $\mathrm{III}$, followed by a standard link of type $\mathrm{II}$, we get the result.
\end{proof}
 
 \begin{lem}\label{Lem:Simple}
 Let $\pi\colon X\to \p^1$ and $\pi'\colon X'\to \p^1$ be two Mori fibrations, where $X,X'$ belong to the list $\F_0,\F_1,\d_6$. Let $\psi\colon X\dasharrow X'$ be a birational map, such that $\pi'\psi=\alpha\pi$ for some $\alpha\in \Aut_\r(\p^1)$. Then, $\psi$ is either an automorphism or $\psi=\varphi_n\cdots \varphi_1$, where each $\varphi_i$ is a sequence of \emph{standard} links of type $\mathrm{II}$. Moreover, if $\psi$ is an isomorphism on the real points $($i.e.\ is an isomorphism $X(\r)\to X'(\r))$, the standard links $\varphi_i$ can also be chosen to be isomorphisms on the real points.
 \end{lem}
 \begin{proof}
 We first show that $\psi=\varphi_n\cdots \varphi_1$, where each $\varphi_i$ is a sequence of links of type $\mathrm{II}$, not necessarily standard. This is done by induction on the number of base-points of $\psi$. If $\psi$ has no base-point, it is an isomorphism. If $q$ is a real proper base-point, or $q,\bar{q}$ are two proper imaginary base-points  (here proper means not infinitely near), we denote by $\varphi_1$ a Sarkisov link of type $\mathrm{II}$ centered at $q$ (or $q,\bar{q}$). Then, $(\varphi_1)^{-1}\psi$ has less base-points than $\psi$. The result follows then by induction. Moreover, if $\psi$ is an isomorphism on the real points, i.e.\ if $\psi$ and $\psi^{-1}$ have no real base-point, then so are all $\varphi_i$.
 
Let $\varphi\colon \d_6\dasharrow \d_6$ be a special link of type $\mathrm{II}$. Then, it is centered at two points $p_1,\bar{p_1}$ lying on the $(-1)$-curves $E_1,\bar{E_1}$. We choose then two general imaginary conjugate points $q_1,\bar{q_1}$, and let $q_2:=\varphi(q_1)$ and $\bar{q_2}:=\varphi(\bar{q_1})$. For $i=1,2$, we denote by $\varphi_i\colon \d_6\dasharrow \d_6$ a standard link centered at $q_i,\bar{q_i}$. The image by $\varphi_2$ of $E_1$ is a curve of self-intersection $1$. Hence, $\varphi_2 \varphi (\varphi_1)^{-1}$ is a standard link of type $\mathrm{II}$.

It remains to consider the case where each $\varphi_i$ is a link $\F_{n_i}\dasharrow \F_{n_{i+1}}$. We denote by $N$ the maximum of the integers $n_i$. If $N\le 1$, we are done because all links of type $\II$ between $\mathbb{F}_i$ and $\mathbb{F}_i'$ with $i,i'\le 1$ are standard. We can thus assume $N\ge 2$, which implies that there exists $i$ such that $n_i=N$, $n_{i-1}<N, n_{i+1}<N$. We choose two general imaginary points $q_1,\bar{q_1}\in \F_{n_{i-1}}$, and write $q_2=\varphi_{i-1}(q_1)$, $q_3=\varphi_{i}(q_2)$. For $i=1,2,3$, we denote by $\tau_i\colon \F_{n_{i}}\dasharrow \F_{n_{i}'}$ a Sarkisov link centered at $q_i,\bar{q_i}$. We obtain then the following commutative diagram
$$\xymatrix@C=1cm{
\F_{n_{i-1}}\ar@{-->}[r]^{\varphi_{i-1}}\ar@{-->}[d]^{\tau_1}& \F_{n_{i}}\ar@{-->}[r]^{\varphi_{i}}\ar@{-->}[d]^{\tau_2}& \F_{n_{i+1}}\ar@{-->}[d]^{\tau_3}\\
\F_{n_{i-1}'}\ar@{-->}[r]^{\varphi_{i-1}'}&\F_{n_{i}'}\ar@{-->}[r]^{\varphi_{i}'}& \F_{n_{i+1}'},
}$$
where $\varphi_{i-1}',\varphi_i'$ are Sarkisov links. By construction, $n_{i-1}',n_i',n_{i+1}'<N$, we can then replace $\varphi_{i}\varphi_{i-1}$ with $(\tau_3)^{-1}\varphi_{i}'\varphi_{i-1}'\tau_1$ and "avoid" $\F_N$. Repeating this process if needed, we end up with a sequence of Sarkisov links passing only through $\F_1$ and $\F_0$. Moreover, since this process does not add any real base-point, it preserves the regularity at real points.\end{proof}
\begin{cor}\label{Cor:DecSimple}
 Let $\pi\colon X\to W$ and $\pi'\colon X'\to W'$ be two Mori fibrations, where $X,X'$ are either $\F_0,\F_1,\d_6$ or $\p^2$.
 Any birational map $\psi\colon X\dasharrow X'$ is either an isomorphism preserving the fibrations or decomposes into links of type $\mathrm{I},\mathrm{III}$, and standard links of type $\mathrm{II}$.
\end{cor}
\begin{proof}
Follows from Proposition~\ref{Prop:Sarkisov}, Lemmas~\ref{Lem:4is123} and~\ref{Lem:Simple}, and the description of Example~\ref{Exa:Sarkisov}.
\end{proof}
 
\section{Generators of the group $\Aut(\p^2(\r))$}\label{AutP2R}
We start this section by describing three kinds of elements of $\Aut(\p^2(\r))$, which are birational maps of $\p^2$ of degree $5$. These maps are associated to three pairs of conjugate imaginary points; the description is then analogue to the description of quadratic maps, which are associated to three points.
\begin{example}\label{Exa:QuinticStd}
Let $p_1,\bar{p_1},p_2,\bar{p_2},p_3,\bar{p_3}\in \p^2$ be  three pairs of imaginary points of $\p^2$, not lying on the same conic. Denote by $\pi\colon X\to \p^2$ the blow-up of the six points, which is an isomorphism $X(\r)\to \p^2(\r)$. Note that $X$ is isomorphic to a smooth cubic of $\p^3$. The set of strict transforms of the conics passing through five of the six points corresponds to  three pairs of imaginary  $(-1)$-curves (or lines on the cubic), and the six curves are disjoint. The contraction of the six curves gives a birational morphism $\eta\colon X\to \p^2$, inducing an isomorphism $X(\r)\to \p^2(\r)$, which contracts the curves onto three pairs of imaginary points $q_1,\bar{q_1},q_2,\bar{q_2},q_3,\bar{q_3}\in \p^2$; we choose the order so that $q_i$ is the image of the conic not passing through $p_i$.
The  map $\psi=\eta\pi^{-1}$ is a birational map $\p^2\dasharrow \p^2$ inducing an isomorphism $\p^2(\r)\to \p^2(\r)$. 

 Let $L\subset \p^2$ be a general line of $\p^2$. The strict transform of $L$ on $X$ by $\pi^{-1}$ has self-intersection $1$ and intersects the six curves contracted by $\eta$ into $2$ points (because these are conics). The image $\psi(L)$ has then six singular points of multiplicity $2$ and self-intersection $25$; it is thus a quintic passing through the $q_i$ with multiplicitiy $2$. The construction of $\psi^{-1}$ being symmetric as the one of $\psi$, the linear system of $\psi$ consists of quintics of $\p^2$ having multiplicity $2$ at  $p_1,\bar{p_1},p_2,\bar{p_2},p_3,\bar{p_3}$. 

One can moreover check that $\psi$ sends the pencil of conics through $p_1,\bar{p_1},p_2,\bar{p_2}$ onto the pencil of conics through $q_1,\bar{q_1},q_2,\bar{q_2}$ (and the same holds for the two other real pencil of conics, through $p_1,\bar{p_1},p_3,\bar{p_3}$ and through $p_2,\bar{p_2},p_3,\bar{p_3}$).
\end{example}

\begin{example}\label{Exa:QuinticSpe}
Let $p_1,\bar{p_1},p_2,\bar{p_2}\in \p^2$ be two pairs of imaginary points of $\p^2$, not on the same line. Denote by $\pi_1\colon X_1\to \p^2$ the blow-up of the four points, and by $E_2,\bar{E_2}\subset X_1$ the curves contracted onto $p_2,\bar{p_2}$ respectively.  Let $p_3\in E_2$ be a point, and $\bar{p_3}\in \bar{E_2}$ its conjugate.  We assume that there is no conic of $\p^2$ passing through $p_1,\bar{p_1},p_2,\bar{p_2},p_3,\bar{p_3}$ and let $\pi_2\colon X_2\to X_1$ be the blow-up of $p_3,\bar{p_3}$.

On $X$, the strict transforms of the two conics $C,\bar{C}$ of $\p^2$, passing through $p_1,\bar{p_1},p_2,\bar{p_2},p_3$ and $p_1,\bar{p_1},p_2,\bar{p_2},\bar{p_3}$ respectively, are imaginary conjugate disjoint $(-1)$ curves. The contraction of these two curves gives a birational morphism $\eta_2\colon X_2\to Y_1$, contracting $C$, $\bar{C}$ onto two points $q_3,\bar{q_3}$. On $Y_1$, we find two pairs of imaginary $(-1)$-curves, all four curves being disjoint. These are the strict transforms of the exceptional curves associated to $p_2,\bar{p_2}$, and of the conics passing through $p_1,p_2,\bar{p_2},p_3,\bar{p_3}$ and $\bar{p_1},p_2,\bar{p_2},p_3,\bar{p_3}$ respectively. The contraction of these curves gives a birational morphism $\eta_1\colon Y_1\to \p^2$, and the images of the four curves are points $q_2,\bar{q_2},q_1,\bar{q_1}$ respectively. Note that the four maps $\pi_1,\pi_2,\eta_1,\eta_2$ are blow-ups of imaginary points, so the birational map $\psi=\eta_1\eta_2(\pi_1\pi_2)^{-1}\colon \p^2\dasharrow \p^2$ induces an isomorphism $\p^2(\r)\to \p^2(\r)$.

In the same way as in Example~\ref{Exa:QuinticStd}, we find that the linear system of $\psi$ is of degree $5$, with multiplicity $2$ at the points $p_i,\bar{p_i}$. The situation is similar for $\psi^{-1}$, with the six points $q_i,\bar{q_i}$ in the same configuration: $q_1,\bar{q_1},q_2,\bar{q_2}$ lie on the plane and $q_3,\bar{q_3}$ are infinitely near to $q_2,\bar{q_2}$ respectively.

One can moreover check that $\psi$ sends the pencil of conics through $p_1,\bar{p_1},p_2,\bar{p_2}$ onto the pencil of conics through $q_1,\bar{q_1},q_2,\bar{q_2}$  and the pencil of conics through  $p_2,\bar{p_2},p_3,\bar{p_3}$ onto the pencil of conics through $q_2,\bar{q_2},q_3,\bar{q_3}$. But, contrary to Example~\ref{Exa:QuinticStd}, there is no pencil of conics through  $q_1,\bar{q_1},q_3,\bar{q_3}$.
\end{example}

\begin{example}\label{Exa:QuinticSpe2}
Let $p_1,\bar{p_1}$ be a pair of two conjugate imaginary points of $\p^2$. We choose a point $p_2$ in the first neighbourhood of $p_1$, and a point $p_3$ in the first neighbourhood of $p_2$, not lying on the exceptional divisor of $p_1$. We denote by $\pi\colon X\to \p^2$ the blow-up of $p_1,\bar{p_1},p_2,\bar{p_2},p_3\bar{p_3}$. We denote by $E_i,\bar{E_i}\subset X$  the irreducible exceptional curves corresponding to the points $p_i,\bar{p_i}$, for $i=1,2,3$. The strict transforms of the two conics through $p_1,\bar{p_1},p_2,\bar{p_2},p_3$ and $p_1,\bar{p_1},p_2,\bar{p_2},\bar{p_3}$ respectively are disjoint $(-1)$-curves on $X$, intersecting the exceptional curves $E_1,\bar{E_1},E_2,\bar{E_2}$ similarly as $E_3,\bar{E_3}$. Hence, there exists a birational morphism $\eta\colon X\to \p^2$ contracting the strict transforms of the two conics and the curves $E_1,\bar{E_1},E_2,\bar{E_2}$.

As in Examples~\ref{Exa:QuinticStd} and \ref{Exa:QuinticSpe}, the linear system of $\psi=\eta\pi^{-1}$ consists of quintics with multiplicity two at the six points $p_1,\bar{p_1},p_2,\bar{p_2},p_3,\bar{p_3}$.
\end{example}

\begin{dfn}
The birational maps of $\p^2$ of degree $5$ obtained in Example~\ref{Exa:QuinticStd} will be called \emph{standard quintic transformations} and those of Example~\ref{Exa:QuinticSpe} and Example~\ref{Exa:QuinticSpe2} will be called \emph{special quintic transformations} respectively.
\end{dfn}

\begin{lem}\label{Lem:AutP2des}
Let $\psi\colon \p^2\dasharrow \p^2$ be a birational map inducing an isomorphism $\p^2(\r)\to \p^2(\r)$. The following hold:
\begin{enumerate}
\item
The degree of $\psi$ is $4k+1$ for some integer $k\ge 0$.
\item
Every multiplicity of the linear system of $\psi$ is even. 
\item
Every curve contracted by $\psi$ is of even degree.
\item
If $\psi$ has degree $1$, it belongs to $\Aut_\r(\p^2)=\PGL(3,\r)$.
\item
If $\psi$ has degree $5$, then it is a standard or special quintic transformation, described in Examples~$\ref{Exa:QuinticStd}$, $\ref{Exa:QuinticSpe}$ or  $\ref{Exa:QuinticSpe2}$, and has thus exactly $6$ base-points.
\item If $\psi$ has at most $6$ base-points, then $\psi$ has degree $1$ or $5$.
\end{enumerate}
\begin{rem}
Part (1) is \cite[Teorema 1]{rv}.
\end{rem}
\end{lem}
\begin{proof}
Denote by $d$ the degree of $\psi$ and by $m_1,\dots,m_k$ the multiplicities of the base-points of $\psi$. The Noether equalities yield $\sum_{i=1}^k m_i=3(d-1)$ and $\sum_{i=1}^k (m_i)^2=d^2-1$.

Let $C,\bar{C}$ be a pair of two curves contracted by $\psi$. Since $C\cap \bar{C}$ does not contain any real point, the degree of $C$ and $\bar{C}$ is even. This yields $(2)$, and implies that all multiplicities of the linear system of $\psi^{-1}$ are even, giving $(2)$.

In particular, $3(d-1)$ is a multiple of $4$ (all multiplicities come by pairs of even integers), which implies that $d=4k+1$ for some integer $k$. Hence $(1)$ is proved.

If the number of base-points is at most $k=6$, then by Cauchy-Schwartz we get 
$$9(d-1)^2=\left(\sum_{i=1}^k m_i\right)^2\le k\sum_{i=1}^k (m_i)^2=k (d^2-1)=6(d^2-1)$$
This yields $9(d-1)\le 6(d+1)$, hence $d\le 5$.

If $d=1$, all $m_i$ are zero, and $\psi\in \Aut_\r(\p^2)$, so we get $(4)$.

If $d=5$, the Noether equalities yield $k=6$ and $m_1=m_2=\dots=m_6=2$. Hence, the base-points of $\psi$ consist of three pairs of conjugate imaginary points $p_1,\bar{p_1},p_2,\bar{p_2},p_3,\bar{p_3}$. Moreover, if a conic passes through $5$ of the six points, its free intersection with the linear system is zero, so it is contracted by $\psi$, and there is no conic through the six points.

$(a)$ If the six points belong to $\p^2$, the map is a standard quintic transformation, described in Example~$\ref{Exa:QuinticStd}$.

$(b)$ If two points are infinitely near, the map is a special quintic transformation, described in Example~$\ref{Exa:QuinticSpe}$.

$(c)$ If four points are infinitely near, the map is a special quintic transformation, described in Example~$\ref{Exa:QuinticSpe2}$.
\end{proof}

Before proving Theorem~\ref{ThmRongaVust}, we will show that all quintic transformations are generated by linear automorphisms and standard quintic transformations:
\begin{lem}\label{Lem:Quinticgenstd}
Every quintic transformation  $\psi\in \Aut(\p^2(\r))$ belongs to the group generated by $\Aut_\r(\p^2)$ and standard quintic transformations.
\end{lem}
\begin{proof}By Lemma~\ref{Lem:AutP2des}, we only need to show the result when $\psi$ is a special quintic transformation as in Example~\ref{Exa:QuinticSpe} or Example~\ref{Exa:QuinticSpe2}.

We first  assume that $\psi$ is a special quintic transformation as in Example~\ref{Exa:QuinticSpe}, with base-points $p_1,\bar{p_1},p_2,\bar{p_2},p_3,\bar{p_3}$, where $p_3,\bar{p_3}$ are infinitely near to $p_2,\bar{p_2}$. For $i=1,2$, we denote by $q_i\in \p^2$ the point which is the image by $\psi$ of the conic passing the five points of  $\{p_1,\bar{p_1},p_2,\bar{p_2},p_3,\bar{p_3}\}\setminus\{p_i\}$.  Then, the base-points of $\psi^{-1}$ are $q_1,\bar{q_1},q_2,\bar{q_2},q_3,\bar{q_3}$, where $q_3$, $\bar{q_3}$ are points infinitely near to $q_2$, $\bar{q_2}$ respectively (see Example~\ref{Exa:QuinticSpe}). 
 We choose a general pair of conjugate imaginary points $p_4,\bar{p_4}\in \p^2$, and write $q_4=\psi(p_4)$, $\bar{q_4}=\psi(\bar{p_4})$. We denote by $\varphi_1$ a standard quintic transformation having base-points at $p_1,\bar{p_1},p_2,\bar{p_2},p_4,\bar{p_4}$, and by $\varphi_2$ a standard quintic transformation having base-points at $q_1,\bar{q_1},q_2,\bar{q_2},q_4,\bar{q_4}$. We now prove that $\varphi_2\psi(\varphi_1)^{-1}$ is a standard quintic transformation; this will yield the result.
 Denote by $p_i',\bar{p_i}'$ the base-points of $(\varphi_1)^{-1}$, with the order associated to the $p_i$, which means that $p_i'$ is the image by $\varphi_i$ of a conic not passing through $p_i$ (see Example~\ref{Exa:QuinticStd}). Similary, we denote by $q_i',\bar{q_i}'$ the base-points of $(\varphi_2)^{-1}$. We obtain the following commutative of birational maps, where the arrows indexed by points are blow-ups of these points:
$$\xymatrix@C=1cm{
& Y_1\ar[ld]_{p_4',\bar{p_4}'}\ar[rd]^{p_4,\bar{p_4}} && Y_2\ar[ld]_{p_3,\bar{p_3}}\ar[rd]^{q_3,\bar{q_3}} && Y_3\ar[ld]_{q_4,\bar{q_4}}\ar[rd]^{q_4',\bar{q_4}'}\\
X_1\ar[d]_(.38){p_1',\bar{p_1}'}_(.62){p_2',\bar{p_2}'}& &X_2\ar@{-->}[ll]_{\hat{\varphi_1}}\ar@{-->}[rr]^{\hat{\psi}}\ar[d]_(.38){p_1,\bar{p_1}}_(.62){p_2,\bar{p_2}}& &X_3\ar@{-->}[rr]^{\hat{\varphi_2}} \ar[d]_(.38){q_1,\bar{q_1}}_(.62){q_2,\bar{q_2}}& &X_4\ar[d]_(.38){q_1',\bar{q_1}'}_(.62){q_2,\bar{q_2}'}\\
\p^2&& \ar@{-->}[ll]_{\varphi_1}\p^2\ar@{-->}[rr]^{\psi} && \p^2\ar@{-->}[rr]^{\varphi_2} && \p^2.
}$$
Each of the surfaces $X_1,X_2,X_3,X_4$ admits a conic bundle structure $\pi_i\colon X_i\to \p^1$, which fibres corresponds to the conics passing through the four points blown-up on $\p^2$ to obtain $X_i$. Moreover, $\hat{\varphi}_1$, $\hat{\psi}$, $\hat{\varphi}_2$ preserve these conic bundle structures. The map $(\hat{\varphi}_1)^{-1}$ blows-up $p_4,\bar{p_4}'$ and contract the fibres associated to them, then $\hat{\psi}$ blows-up $p_3,\bar{p_3}$ and contract the fibres associated to them. The map $\hat{\varphi}_2$ blow-ups the points $q_4,\bar{q_4}$, which correspond to the image of the curves contracted by $(\hat{\varphi}_1)^{-1}$, and contracts their fibres, corresponding to the exceptional divisors of the points $p_4,\hat{p_4}'$. Hence, $\hat{\varphi}_2\hat{\psi}\hat{\varphi}_1$ is the blow-up of two imaginary points $p_3',\hat{p_3}'\in X_1$, followed by the contraction of their fibres. We obtain the following commutative diagram
$$\xymatrix@C=1cm{
& Z\ar[ld]_{p_3',\bar{p_3}'}\ar[rd]^{q_3',\bar{q_3}'}\\
X_1\ar@{-->}[rr]^{\hat{\varphi_2}\hat{\psi}(\hat{\varphi_1})^{-1}}\ar[d]_(.38){p_1',\bar{p_1}'}_(.62){p_2',\bar{p_2}'}& &X_4\ar[d]_(.38){q_1',\bar{q_1}'}_(.62){q_2,\bar{q_2}'}\\
\p^2\ar@{-->}[rr]^{{\varphi_2}{\psi}({\varphi_1})^{-1}}&& \p^2,
}$$
and the points $p_3',\bar{p_3}'$ correspond to point of $\p^2$, hence $\varphi_2\psi(\varphi_1)^{-1}$ is a standard quintic transformation.

The remaining case is when $\psi$ is a special quintic transformation as in Example~\ref{Exa:QuinticSpe2}, with base-points with base-points $p_1,\bar{p_1},p_2,\bar{p_2},p_3,\bar{p_3}$, where $p_3,\bar{p_3}$ are infinitely near to $p_2,\bar{p_2}$ and these latter are infinitely near to $p_1,\bar{p_1}$. The map $\psi^{-1}$ has base-points $q_1,\bar{q_1},q_2,\bar{q_2},q_3,\bar{q_3}$, having the same configuration (see Example~\ref{Exa:QuinticSpe2}).
 We choose a general pair of conjugate imaginary points $p_4,\bar{p_4}\in \p^2$, and write $q_4=\psi(p_4)$, $\bar{q_4}=\psi(\bar{p_4})$. We denote by $\varphi_1$ a special quintic transformation having base-points at $p_1,\bar{p_1},p_2,\bar{p_2},p_4,\bar{p_4}$, and by $\varphi_2$ a special quintic transformation having base-points at $q_1,\bar{q_1},q_2,\bar{q_2},q_4,\bar{q_4}$. The maps $\varphi_1,\varphi_2$ have four proper base-points, and are thus given in Example~\ref{Exa:QuinticSpe}. The same proof as before implies that $\varphi_2\psi(\varphi_1)^{-1}$ is a special quintic transformation with four base-points. This gives the result.\end{proof}

\begin{lem}\label{Lem:FiveLinksDegree5}
Let $\varphi\colon \p^2\dasharrow \p^2$ be a birational map, that decomposes as $\varphi=\varphi_5\cdots \varphi_1$, where $\varphi_i\colon X_{i-1}\dasharrow X_i$ is a Sarkisov link for each $i$, where $X_0=\p^2$, $X_1=Q_{3,1}$, $X_2=X_3=\d_6$, $X_4=\s_2$, $X_5=\p^2$. If $\varphi_2$ is an automorphism of $\d_6(\r)$ and $\varphi_4\varphi_3\varphi_2$ sends the base-point of $(\varphi_1)^{-1}$ onto the base-point of $\varphi_5$, then $\varphi$ is an automorphism of $\p^2(\r)$ of degree $5$.
\end{lem}

\begin{proof}
We  have the following commutative diagram, where each $\pi_i$ is the blow-up of two conjugate imaginary points and each $\eta_i$ is the blow-up of one real point. The two maps $(\varphi_2)^{-1}$ and $\varphi_4$ are also blow-ups of imaginary points.
$$
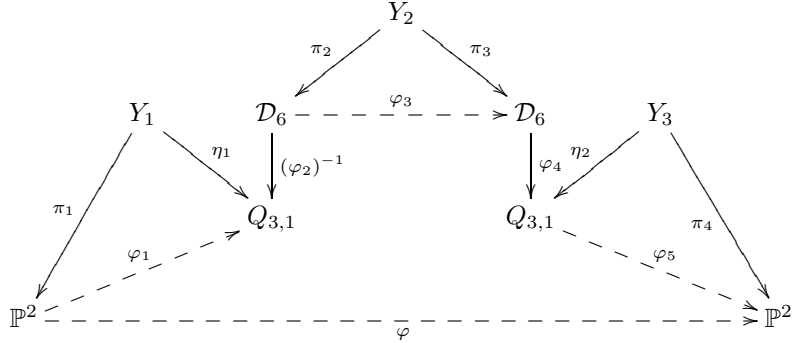
\xymatrix@C=1cm{
&&&Y_2\ar[ld]_{\pi_2}\ar[rd]^{\pi_3}\\
&Y_1\ar[ldd]_{\pi_1}\ar[rd]^{\eta_1} & \d_6\ar[d]^{(\varphi_2)^{-1}}\ar@{-->}[rr]^{\varphi_3}&&\d_6\ar[d]^{\varphi_4}& Y_3\ar[ld]_{\eta_2}\ar[rdd]_{\pi_4}\\
&&Q_{3,1} && Q_{3,1}\ar@{-->}[rrd]^{\varphi_5}\\
\p^2\ar@{-->}[rru]^{\varphi_1}\ar@{-->}[rrrrrr]_{\varphi}&&&&&&\p^2.
}$$
The only real base-points are those blown-up by $\eta_1$ and $\eta_2$.
Since $\eta_2$ blows-up the image by $\varphi_4\varphi_3\varphi_2$ of the real point blown-up by $\eta_1$, the map $\varphi$ has at most $6$ base-points, all being imaginary, and the same holds for $\varphi^{-1}$. Hence, $\varphi$ is an automorphism of $\p^2(\r)$ with at most $6$ base-points. We can moreover see that $\varphi\not\in \Aut_\r(\p^2)$, since the two curves of $Y_2$ contracted by $\pi_2$ are sent by $\varphi_4\pi_3$ onto conics of $Q_{3,1}$, which are therefore not contracted by $\varphi_5$.

Lemma~\ref{Lem:AutP2des} implies that $\psi$ has degree $5$.
\end{proof}

\begin{prop}\label{Prop:Gen5}
The group $\Aut(\p^2(\r))$ is generated by  $\Aut_\r(\p^2)$ and by elements of  $\Aut(\p^2(\r))$ of degree $5$.
\end{prop}
\begin{proof}
Let us prove that any $\varphi\in \Aut(\p^2(\r))$ is generated by $\Aut_\r(\p^2)$ and elements of  $\Aut(\p^2(\r))$ of degree $5$ . We can assume that $\varphi\not\in \Aut_\r(\p^2)$, decompose it into Sarkisov links: $\varphi=\varphi_r\circ \cdots \circ \varphi_1$.

It follows from the construction of the links (see \cite{bib:Corti}) that if $\varphi_i$ is of type $\mathrm{I}$ or $\mathrm{II}$, each base-point of $\varphi_i$ is a base-point of $\varphi_r\dots\varphi_i$ (and in fact of maximal multiplicity). 

We proceed by induction on $r$.

Since $\varphi$ has no real base-point, the first link $\varphi_1$ is then of type $\mathrm{II}$ from $\p^2$ to $Q_{3,1}$, and $\varphi_r \cdots  \varphi_2$ has a unique real base-point $r\in Q_{3,1}$, which is the base-point of $(\varphi_1)^{-1}$. If $\varphi_2$ blows-up this point, then $\varphi_2\varphi_1\in \Aut_\r(\p^2)$. We can then assume that $\varphi_2$ is a link of type $\mathrm{I}$ from $Q_{3,1}$ to $\d_6$. The map $\varphi_2\varphi_1$ can be written as $\eta\pi^{-1}$, where $\pi\colon X\to \p^2$ is the blow-up of two pairs of imaginary points, say $p_1,\bar{p_1},p_2,\bar{p_2}$ and $\eta\colon X\to \d_6$ is the contraction of the strict transform of the real line passing through $p_1,\bar{p_1}$, onto a real point $q\in \d_6$.  Note that $p_1,\bar{p_1}$ are proper points of $\p^2$, blown-up by $\varphi_1$ and $p_2,\bar{p_2}$ either are proper base-points or are infinitely near to $p_1,\bar{p_1}$.

The fibration $\d_6\to \p^1$ corresponds to conics through $p_1,\bar{p_1},p_2,\bar{p_2}$.  If $\varphi_3$ is a link of type $\mathrm{III}$, then $\varphi_3\varphi_2$ is an automorphism of $Q_{3,1}$ and we decrease $r$. Then, $\varphi_3$ is of type $\mathrm{II}$. If $q$ is a base-point of $\varphi_3$, then $\varphi_3=\eta'\eta^{-1}$, where $\eta'\colon X\to \d_6$ is the contraction of  the strict transform of the line through $p_2,\bar{p_2}$. We can then write $\varphi_3\varphi_2\varphi_1$ into only two links, exchanging $p_1$ with $p_2$ and $\bar{p_1}$ with $\bar{p_2}$. The remaining case is when $\varphi_3$ is the blow-up of two imaginary points $p_3,\bar{p_3}$ of $\d_6$, followed by the contraction of the strict transforms of their fibres.

We denote by $q'\in \d_6(\r)$ the image of $q$ by $\varphi_3$, consider $\psi_4=(\varphi_2)^{-1}\colon \d_6\to Q_{3,1}$, which is a link of type $\mathrm{III}$, and write $\psi_5\colon Q_{3,1}\dasharrow \p^2$ the stereographic projection by $\psi_4(q')$, which is a link of type $\mathrm{II}$ centered at $\psi_4(q')$. By Lemma~\ref{Lem:FiveLinksDegree5}, the map $\chi=\psi_5\psi_4\varphi_3\varphi_2\varphi_1$ is an element of $\Aut(\p^2(\r))$ of degree $5$. Since $\varphi\chi^{-1}$ decomposes into one link less than $\varphi$, this concludes the proof by induction.\end{proof}
\begin{proof}[Proof of Theorem~$\ref{ThmRongaVust}$]
By Proposition~\ref{Prop:Gen5}, $\Aut(\p^2(\r))$ is generated by  $\Aut_\r(\p^2)$ and by elements of  $\Aut(\p^2(\r))$ of degree $5$. Thanks to Lemma~\ref{Lem:Quinticgenstd}, $\Aut(\p^2(\r))$ is indeed generated by projectivities and standard quintic transformations.\end{proof}

\section{Generators of the group $\Bir_\r(\p^2)$}\label{BirP2R}

\begin{lem}\label{lem:Exchange}
Let $\varphi\colon Q_{3,1}\dasharrow Q_{3,1}$ be a birational map, that decomposes as $\varphi=\varphi_3\varphi_2\varphi_1$, where $\varphi_i\colon X_{i-1}\dasharrow X_i$ is a Sarkisov link for each $i$, where $X_0=Q_{3,1}=X_2$, $X_1=\d_6$. If $\varphi_2$ has a real base-point, then, $\varphi$ can be written as $\varphi=\psi_2\psi_1$, where $\psi_1,(\psi_2)^{-1}$ are links of type $\mathrm{II}$ from $Q_{3,1}$ to $\p^2$.
\end{lem}
\begin{proof}
We  have the following commutative diagram, where each of the maps $\eta_1,\eta_2$ blow-ups a real point, and each of the maps $(\varphi_1)^{-1},\varphi_3$ is the blow-up of two conjugate imaginary points. 
$$\xymatrix@C=1cm{
&&&Y\ar[ld]_{\eta_1}\ar[rd]^{\eta_2}\\
&& \d_6\ar[d]^{(\varphi_1)^{-1}}\ar@{-->}[rr]^{\varphi_2}&&\d_6\ar[d]^{\varphi_3}& \\
&&Q_{3,1} \ar@{-->}[rr]^{\varphi}&& Q_{3,1}.
}$$
The map $\varphi$ has thus exactly three base-points, two of them being imaginary and one being real; we denote them by $p_1,\bar{p_1}$, $q$.
The fibres of the Mori fibration $\d_6\to \p^1$ correspond to conics of $Q_{3,1}$ passing through the points $p_1,\bar{p_1}$. The real curve contracted by $\eta_2$ is thus the strict transform of the conic $C$ of $Q_{3,1}$ passing through $p_1,\bar{p_1}$ and $q$. The two curves contracted by $\varphi_3$ are the two imaginary sections of self-intersection $-1$, which corresponds to the strict transforms of the two imaginary lines $L_1,L_2$ of $Q_{3,1}$ passing through $q$.

We can then decompose $\varphi$ as the blow-up of $p_1,p_2,q$, followed by the contraction of the strict transforms of $C,L_1,L_2$. Denote by $\psi_1\colon Q_{3,1}\dasharrow \p^2$ the link of type $\mathrm{II}$ centered at $q$, which is the blow-up of $q$ followed by the contraction of the strict transform of $L_1,L_2$, or equivalenty the stereographic projection centered at $q$. The curve $\psi_1(C)$ is a real line of $\p^2$, which contains the two points $\psi_1(p_1)$, $\psi_1(\bar{p_1})$. The map $\psi_2=\varphi(\psi_1)^{-1}\colon \p^2\dasharrow Q_{3,1}$ is then the blow-up of these two points, followed by the contraction of the line passing through both of them. It is then a link of type $\mathrm{II}$.
\end{proof}

\begin{proof}[Proof of Theorem~$\ref{thm:BirP2}$]
Let us prove that any $\varphi\in \Bir_\r(\p^2)$ is in the group generated by $\Aut_\r(\p^2)$, $\sigma_0$, $\sigma_1$, and standard quintic transformations of $\p^2$. We can assume that $\varphi\not\in \Aut_\r(\p^2)$, decompose it into Sarkisov links: $\varphi=\varphi_r\circ \cdots \circ \varphi_1$. By Corollary~\ref{Cor:DecSimple}, we can assume that all the $\varphi_i$ are  links of type $\mathrm{I},\mathrm{III}$, or standard links of type $\mathrm{II}$.

We proceed by induction on $r$, the case $r=0$ being obvious.

Note that $\varphi_1$ is either a link of type $\mathrm{I}$ from $\p^2$ to $\mathbb{F}_1$, or a link of type $\mathrm{II}$ from $\p^2$ to $Q_{3,1}$. We now study the possibilities for the base-points of $\varphi_1$ and the next links:

$(1)$ Suppose that $\varphi_1\colon \p^2\dasharrow \F_1$ is a link of type $\mathrm{I}$, and that $\varphi_2$ is a link $\F_1\dasharrow \F_1$. Then, $\varphi_2$ blows-up two imaginary base-points of $\F_1$, not lying on the exceptional curve. Hence, $\psi=(\varphi_1)^{-1}\varphi_2\varphi_1$ is a quadratic transformation of $\p^2$ with three proper base-points, one real and two imaginary. It is thus equal to $\alpha\sigma_1\beta$ for some $\alpha,\beta\in \Aut_\r(\p^2)$. Replacing $\varphi$ with $\varphi\psi^{-1}$, we obtain a decomposition with less Sarkisov links, and conclude by induction.

$(2)$  Suppose that $\varphi_1\colon \p^2\dasharrow \F_1$ is a link of type $\mathrm{I}$, and that $\varphi_2$ is a link $\F_1\dasharrow \F_0$. Then, $\varphi_2\varphi_1$ is the blow-up of two real points $p_1,p_2$ of $\p^2$ followed by the contraction of the line through $p_1,p_2$. The exceptional divisors of $p_1,p_2$ are two $(0)$-curves of $\F_0=\p^1\times \p^1$, intersecting at one real point. 

$(2a)$ Suppose first that $\varphi_3$ has a base-point which is real, and not lying on $E_1,E_2$. Then, $\psi=(\varphi_1)^{-1}\varphi_3\varphi_2\varphi_1$ is a quadratic transformation of $\p^2$ with three proper base-points, all real. It is thus equal to $\alpha\sigma_0\beta$ for some $\alpha,\beta\in \Aut_\r(\p^2)$. Replacing $\varphi$ with $\varphi\psi^{-1}$, we obtain a decomposition with less Sarkisov links, and conclude by induction. 

$(2b)$ Suppose now that $\varphi_3$ has imaginary base-points, which are $q,\bar{q}$. Since $\varphi_3$ is a standard link of type $\mathrm{II}$, it goes from $\F_0$ to $\F_0$, so $q$ and $\bar{q}$ do not lie on a $(0)$-curve, and then do not belong to the curves $E_1,E_2$.
We can then decompose $\varphi_2\varphi_3\colon \F_1\dasharrow \F_2$ into a Sarkisov link centered at two imaginary points, followed by a Sarkisov link centered at a real point. This reduces to case $(1)$, already treated before.

$(2c)$ The remaining case (for $(2)$) is when $\varphi_3$ has a base-point $p_3$ which is real, but lying on $E_1$ or $E_2$. We choose a general real point $p_4\in \F_0$, and denote by $\theta\colon \F_0\dasharrow \F_1$ a Sarkisov link centered at $p_4$. We observe that $\psi=(\varphi_1)^{-1}\theta\varphi_2\varphi_1$ is a quadratic map as in case $(2a)$, and that $\varphi\psi^{-1}=\varphi_n\dots\varphi_3\theta^{-1}\varphi_1$ admits now a decomposition of the same length, but which is in case $(2a)$.

$(3)$ Suppose now that $\varphi_1\colon \p^2\dasharrow Q_{3,1}$ is a link of type $\mathrm{II}$ and that $\varphi_2$ is a link of type $\mathrm{II}$ from $S$ to $\p^2$. If  $\varphi_2$ and $(\varphi_1)^{-1}$ have the same real base-point, the map $\varphi_2\varphi_1$ belongs to $\Aut_\r(\p^2)$. Otherwise, $\varphi_2\varphi_1$ is a quadratic map with one unique real base-point $q$ and two imaginary base-points. It is then equal to  $\alpha\sigma_0\beta$ for some $\alpha,\beta\in \Aut_\r(\p^2)$. We conclude as before by induction hypothesis.

$(4)$ Suppose that $\varphi_1\colon \p^2\dasharrow Q_{3,1}$ is a link of type $\mathrm{II}$ and $\varphi_2$ is a link of type $\mathrm{I}$ from $S$ to $\d_6$. 
If $\varphi_3$ is a Sarkisov link of type $\mathrm{III}$, then $\varphi_3\varphi_2$ is an automorphism of $Q_{3,1}$, so we can decrease the length. We only need to consider the case where $\varphi_3$ is a link of type $\mathrm{II}$ from $\d_6$ to $\d_6$. If $\varphi_3$ has a real base-point, we apply Lemma~\ref{lem:Exchange} to write $(\varphi_2)^{-1}\varphi_3\varphi_2=\psi_2\psi_1$ where $\psi_1,(\psi_2)^{-1}$ are links  $Q_{3,1}\dasharrow \p^2$. By $(3)$, the map $\chi=\psi_1\varphi_1$ is generated by $\Aut_\r(\p^2)$ and $\sigma_0$. We can then replace $\varphi$ with $\varphi\chi^{-1}=\varphi_r\cdots \varphi_3\varphi_2(\psi_1)^{-1}=\varphi_r\cdots \varphi_4\varphi_2\psi_2$, which has a shorter decomposition. The last case is when $\varphi_3$ has two imaginary base-points. We denote by $q\in Q_{3,1}$ the real base-point of $(\varphi_1)^{-1}$, write $q'=(\varphi_2)^{-1}\varphi_3\varphi_2(q)\in Q_{3,1}$ and denote by $\psi\colon Q_{3,1}\dasharrow \p^2$ the stereographic projection centered at $q'$. By Lemma~\ref{Lem:FiveLinksDegree5}, the map $\chi=\psi(\varphi_2)^{-1}\varphi_3\varphi_2\varphi_1$ is an automorphism of $\p^2(\r)$ of degree $5$, which is generated by $\Aut_\r(\p^2)$ and standard automorphisms of $\p^2(\r)$ of degree $5$ (Lemma~\ref{Lem:Quinticgenstd}). We can thus replace $\varphi$ with $\varphi\chi^{-1}$, which has a decomposition of shorter length.
\end{proof}

\section{Generators of the group $\Aut(Q_{3,1}(\r))$}\label{Q31}

\begin{example}\label{Exa:CubicStd}
Let $p_1,\bar{p_1},p_2,\bar{p_2}\in Q_{3,1}\subset \p^3$ be two pairs of conjugate imaginary points, not on the same plane of $\p^3$. Let $\pi\colon X\to Q_{3,1}$ be the blow-up of these points. The imaginary plane of $\p^3$ passing through $p_1,\bar{p_2},\bar{p_2}$ intersects $Q_{3,1}$ onto a conic, having self-intersection $2$: two general different conics on $Q_{3,1}$ are the trace of hyperplanes, and intersect then into two points, being on the line of intersection of the two planes. The strict transform of this conic on $S$ is thus a $(-1)$-curve on $S$. Doing the same for the other conics passing through $3$ of the points $p_1,\bar{p_1},p_2,\bar{p_2}$, we obtain four disjoint $(-1)$-curves on $X$, that we can contract in order to obtain a birational morphism $\eta\colon X\to Q_{3,1}$; note that the target is $Q_{3,1}$ because it is a smooth projective rational surface of Picard rank $1$. We obtain then a birational map $\psi=\eta\pi^{-1}\colon Q_{3,1}\dasharrow Q_{3,1}$ inducing an isomorphism $Q_{3,1}(\r)\to Q_{3,1}(\r)$.

Denote by $H\subset Q_{3,1}$ a general hyperplane section. The strict transform of $H$ on $X$ by $\pi^{-1}$ has self-intersection $2$ and has intersection $2$ with the $4$ curves contracted. The image $\psi(H)$ has thus multiplicity $2$  and   self-intersection $18$; it is then the trace of a cubic section. The construction of $\psi$ and $\psi^{-1}$ being similar, the linear system of $\psi$ consists of cubic sections with multiplicity $2$ at $p_1,\bar{p_1},p_2,\bar{p_2}$.

\end{example}

\begin{example}\label{Exa:CubicSpe}
Let $p_1,\bar{p_1}\in Q_{3,1}\subset \p^3$ be two conjugate imaginary points and let $\pi_1\colon X_1\to Q_{3,1}$ be the blow-up of the two points. Denote by $E_1,\bar{E_1}\subset X_1$ the curves contracted onto $p_1,\bar{p_1}$ respectively.  Let $p_2\in E_1$ be a point, and $\bar{p_2}\in \bar{E_1}$ its conjugate.  We assume that there is no conic of $Q_{3,1}\subset \p^3$ passing through $p_1,\bar{p_1},p_2,\bar{p_2}$ and let $\pi_2\colon X_2\to X_1$ be the blow-up of $p_2,\bar{p_2}$.

On $X$, the strict transforms of the two conics $C,\bar{C}$ of $\p^2$, passing through $p_1,\bar{p_1},p_2$ and $p_1,\bar{p_1},\bar{p_2}$ respectively, are imaginary conjugate disjoint $(-1)$ curves. The contraction of these two curves gives a birational morphism $\eta_2\colon X_2\to Y_1$. On this latter surface, we find two disjoint conjugate imaginary $(-1)$-curves. These are the strict transforms of the exceptional curves associated to $p_1,\bar{p_1}$. The contraction of these curves gives a birational morphism $\eta_1\colon Y_1\to Q_{3,1}$.  The birational map $\psi=\eta_1\eta_2(\pi_1\pi_2)^{-1}\colon Q_{3,1}\dasharrow Q_{3,1}$ induces an isomorphism $Q_{3,1}(\r)\to Q_{3,1}(\r)$.

\end{example}

\begin{dfn}
The birational maps of $Q_{3,1}$ of degree $3$ obtained in Example~\ref{Exa:CubicStd} will be called \emph{standard cubic transformations} and those of Example~\ref{Exa:CubicSpe} will be called \emph{special cubic transformations}. 
\end{dfn}

Note that since $\Pic(Q_{3,1})=\mathbb{Z}H$, where $H$ is an hyperplane section, we can associate to any birational map $Q_{3,1}\dasharrow Q_{3,1}$, an integer $d$, which is the \emph{degree of the map}, such that $\psi^{-1}(H)=dH$.
\begin{lem}
Let $\psi\colon Q_{3,1}\dasharrow Q_{3,1}$ be a birational map inducing an isomorphism $Q_{3,1}(\r)\to Q_{3,1}(\r)$. The following hold:
\begin{enumerate}
\item
The degree of $\psi$ is $2k+1$ for some integer $k\ge 0$.
\item
If $\psi$ has degree $1$, it belongs to $\Aut_\r(Q_{3,1})=\PO(3,1)$.
\item
If $\psi$ has degree $3$, then it is a standard or special cubic transformation, described in Examples~$\ref{Exa:CubicStd}$ and $\ref{Exa:CubicSpe}$, and has thus exactly $4$ base-points.
\item
If $\psi$ has at most $4$ base-points, then $\psi$ has degree $1$ or $3$.
\end{enumerate}
\end{lem}
\begin{proof}
Denote by $d$ the degree of $\psi$ and by $a_1,\dots,a_n$ the multiplicities of the base-points of $\psi$. Denote by  $\pi\colon X\to Q_{3,1}$ the blow-up of the base-points, and by $E_1,\dots,E_n\in \Pic(X)$ the divisors being the total pull-back of the exceptional $(-1)$-curves obtained after blowing-up the points. Writing $\eta\colon X\to Q_{3,1}$ the birational morphism $\psi\pi$, we obtain 
$$\begin{array}{rcl}
\eta^{*}(H)&=&d\pi^*(H)-\sum_{i=1}^n a_i E_i\\
K_X&=&\pi^{*}(-2H)+\sum_{i=1}^n E_i.\end{array}$$ 
Since $H$ corresponds to a smooth rational curve of self-intersection $2$, we have $(\eta^{*}(H))^2$ and $\eta^{*}(H)\cdot K_X=-4$. We find then 
$$\begin{array}{rcrcl}
2&=&(\eta^{*}(H))^2&=&2d^2-\sum_{i=1}^n (a_i)^2\\
4&=&-K_X \cdot \eta^{*}(H)&=&4d-\sum_{i=1}^n a_i.
\end{array}$$
Since multiplicities come by pairs, $n=2m$ for some integer $m$ and we can order the $a_i$ so that $a_{i}=a_{n+1-i}$ for $i=1,\dots,m$. This yields
$$\begin{array}{rcl}
d^2-1&=&\sum_{i=1}^m (a_i)^2\\
2(d-1)&=&\sum_{i=1}^m a_i
\end{array}$$
Since $(a_i)^2\equiv a_i\pmod 2$, we find $d^2-1\equiv 2(d-1)\equiv0 \pmod{2}$, hence $d$ is odd. This gives $(1)$.

If the number of base-points is at most $4$, we can choose $m=2$, and obtain by Cauchy-Schwartz
$$4(d-1)^2=\left(\sum_{i=1}^m a_i\right)^2\le m\sum_{i=1}^m (a_i)^2=m (d^2-1)=2(d^2-1).$$
 This yields $2(d-1)\le d+1$, hence $d\le 3$.

If $d=1$, all $a_i$ are zero, and $\psi\in \Aut_\r(Q_{3,1})$.

If $d=3$, we get $\sum_{i=1}^m (a_i)^2=8$, $\sum_{i=1}^m a_i=4$, so  $m=2$ and $a_1=a_2=2$. Hence, the base-points of $\psi$ consist of two pairs of conjugate imaginary points $p_1,\bar{p_1},p_2,\bar{p_2}$. Moreover, if a conic passes through $3$ of the points, its free intersection with the linear system is zero, so it is contracted by $\psi$, and there is no conic through the four points.

$(a)$ If the four points belong to $Q_{3,1}$, the map is a standard cubic transformation, described in Example~$\ref{Exa:CubicStd}$.

$(b)$ If two points are infinitely near, the map is a special cubic transformation, described in Example~$\ref{Exa:CubicSpe}$.
\end{proof}
\begin{lem}\label{Lem:ThreeLinksDegree3}
Let $\varphi\colon Q_{3,1}\dasharrow Q_{3,1}$ be a birational map, that decomposes as $\varphi=\varphi_3\varphi_2\varphi_1$, where $\varphi_i\colon X_{i-1}\dasharrow X_i$ is a Sarkisov link for each $i$, where $X_0=Q_{3,1}=X_2$, $X_1=\d_6$. If $\varphi_2$ is an automorphism of $\d_6(\r)$ then $\varphi$ is a cubic automorphism of $Q_{3,1}(\r)$ of degree $3$ described in  in Examples~$\ref{Exa:CubicStd}$ and~$\ref{Exa:CubicSpe}$. Moreover, $\varphi$ is a standard cubic transformation if and only if the link $\varphi_2$ of type $\mathrm{II}$ is a standard link of type $\mathrm{II}$.
\end{lem}
\begin{proof}
We  have the following commutative diagram, where each of the maps $\pi_1,\pi_2,(\varphi_1)^{-1},\varphi_3$ is the blow-up of two conjugate imaginary points. 
$$\xymatrix@C=1cm{
&&&Y\ar[ld]_{\pi_1}\ar[rd]^{\pi_2}\\
&& \d_6\ar[d]^{(\varphi_1)^{-1}}\ar@{-->}[rr]^{\varphi_2}&&\d_6\ar[d]^{\varphi_3}& \\
&&Q_{3,1} \ar@{-->}[rr]^{\varphi}&& Q_{3,1}.
}$$
Hence, $\varphi$ is an automorphism of $\p^2(\r)$ with at most $4$ base-points. We can moreover see that $\varphi\not\in \Aut_\r(Q_{3,1})$, since the two curves of $Y_2$ contracted blown-up by $\pi_2$ are sent by $\varphi_3\pi_3$ onto conics of $Q_{3,1}$, contracted by $\varphi^{-1}$.

Lemma~\ref{Lem:AutP2des} implies that $\varphi$ is cubic automorphism of $Q_{3,1}(\r)$ of degree $3$ described in  in Examples~$\ref{Exa:CubicStd}$ and~$\ref{Exa:CubicSpe}$. In particular, $\varphi$ has exactly four base-points, blown-up by $(\varphi_1)^{-1}\pi_1$. Moreover, $\varphi$ is a standard cubic transformation if and only these four points are proper base-points of $Q_{3,1}$. This corresponds to saying that the two base-points of $\varphi_2$ do not belong to the exceptional curves contracted by $(\varphi_1)^{-1}$, and is thus the case exactly when $\varphi_2$ is a standard link of type $\mathrm{II}$.
\end{proof}

\begin{proof}[Proof of Theorem~$\ref{Thmkm}$]
Let us prove that any $\varphi\in \Aut(Q_{3,1}(\r))$ is generated by $\Aut_\r(Q_{3,1})$ and standard cubic transformations of  $\Aut(Q_{3,1}(\r))$ of degree $3$. We can assume that $\varphi\not\in \Aut_\r(\p^2)$, decompose it into Sarkisov links: $\varphi=\varphi_r\circ \cdots \circ \varphi_1$.

As already explained in the proof of Proposition~\ref{Prop:Gen5}, if $\varphi_i$ is of type $\mathrm{I}$ or $\mathrm{II}$, each base-point of $\varphi_i$ is a base-point of $\varphi_r\dots\varphi_i$. This implies that all links are either of type $\mathrm{I}$, from $Q_{3,1}$ to $\d_6$, of type $\mathrm{II}$ from $\d_6$ to $\d_6$ with imaginary base-points, or of type $\mathrm{III}$ from $\d_6$ to $Q_{3,1}$. Moreover, by Lemma~\ref{Lem:Simple}, we can assume that all links of type $\mathrm{II}$ are standard.

We proceed by induction on $r$.

Since $\varphi$ has no real base-point, the first link $\varphi_1$ is then of type $\mathrm{I}$ from $Q_{3,1}$ to $\d_6$. If $\varphi_2$ is of type $\mathrm{III}$, then $\varphi_2\varphi_1\in \Aut_\r(Q_{3,1})$. We replace these two links and conclude by induction. If $\varphi_2$ is a standard link of type $\mathrm{II}$, then $\psi=(\varphi_1)^{-1}\varphi_2\varphi_1$ is a standard cubic transformation. Replacing $\varphi$ with $\varphi\psi^{-1}$ decreases the number of links, so we conclude by induction.
\end{proof}

\subsubsection*{Twisting maps and factorisation}
Choose a real line $L\subset \p^3$, which does not meet $Q_{3,1}(\r)$. The projection from $L$ gives a morphism $\pi_L\colon Q_{3,1}(\r)\to \p^1(\r)$, which induces a conic bundle structure on the blow-up $\tau_L\colon \d_6\to Q_{3,1}$ of the two imaginary points of $L\cap Q_{3,1}$.

We denote by $T(Q_{3,1},\pi_L)\subset \Aut(Q_{3,1}(\r))$ the group of elements $\varphi\in \Aut(Q_{3,1}(\r))$ such that $\pi_L\varphi=\pi_L$ and such that the lift $(\tau_L)^{-1}\varphi\tau_L\in \Aut(\d_6(\r))$ preserves the set of two imaginary $(-1)$-curves which are sections of the conic bundle $\pi_L\tau_L$.

Any element $\varphi\in T(Q_{3,1},\pi_L)$ is called a \emph{twisting map of $Q_{3,1}$ with axis $L$}. 

Choosing the line $w=x=0$ for $L$, we can get the more precise description given in \cite{hm3,km1}: the twisting maps corresponds in local coordinates $(x,y,z)\mapsto (1:x:y:z)$ to 
$$
\varphi_M \colon 
(x,y,z)\mapsto \bigl(x,(y,z)\cdot M(x)\bigr)
$$
where  $M\colon[-1,1] \to O(2)\subset \PGL(2,\r)=\Aut(\p^1)$ is a real algebraic map.

\begin{prop} 
Any twisting map with axis $L$ is a composition of twisting maps with axis $L$, of degree $1$ and $3$.
\end{prop}
\begin{proof}
We can asume that $L$ is the line $y=z=0$.

The projection $\tau\colon \d_6\to Q_{3,1}$ is a link of type $\III$, described in Example~$\ref{Exa:Sarkisov}(3)$, which blows-up two imaginary points of $Q_{3,1}$. The fibres of the Mori Fibration $\pi\colon \d_6\to \p^1$ correspond then, via $\tau$, to the fibres of $\pi_L\colon Q_{3,1}(\r)\to \p^1(\r)$. Hence, a twisting map of $Q_{3,1}$ corresponds to a map of the form $\tau\varphi\tau^{-1}$, where $\varphi\colon \d_6\dasharrow \d_6$ is a birational map such that $\pi\varphi=\pi$, and which preserves the set of two $(-1)$-curves. This implies that $\varphi$ has all its base-points on the two $(-1)$-curves. It remains to argue as in Lemma~\ref{Lem:Simple}, and decompose $\varphi$ into links that have only base-points on the set of two $(-1)$-curves.
\end{proof}

\section{Generators of the group $\Aut(\F_0(\r))$}\label{F0}
\begin{proof}[Proof of Theorem~$\ref{ThmTorus}$]
Let us prove that any $\varphi\in \Aut(\F_0(\r))$ is generated by $\Aut_\r(\F_0)$ and by the the involution
$$\tau_0\colon ((x_0:x_1),(y_0:y_1))\dasharrow ((x_0:x_1),(x_0y_0+x_1y_1:x_1y_0-x_0y_1)).$$
Observe that $\tau_0$ is a Sarkisov link $\F_0\dasharrow \F_0$ that is the blow-up of the two imaginary points $p=((\im:1),(\im:1))$, $\bar{p}=((-\im:1),(-\im:1))$, followed by the contraction of the two fibres of the first projection $\F_0\to \p^1$ passing through $p,\bar{p}$.

 We can assume that $\varphi\not\in \Aut_\r(\p^2)$, decompose it into Sarkisov links: $\varphi=\varphi_r\circ \cdots \circ \varphi_1$.  
As already explained, if $\varphi_i$ is of type $\mathrm{I}$ or $\mathrm{II}$, each base-point of $\varphi_i$ is a base-point of $\varphi_r\dots\varphi_i$. This implies that all links are either of type $\mathrm{IV}$, or of $\mathrm{II}$, from $\F_{2d}$ to $\F_{2d'}$, with exactly two imaginary base-points. Moreover, by Lemma~\ref{Lem:Simple}, we can assume that all links of type $\mathrm{II}$ are standard, so all go from $\F_0$ to $\F_0$.
 
 Each link of type $\mathrm{IV}$ is an element of $\Aut_\r(\F_0)$. 
 
 Each link $\varphi_i$ of type $\mathrm{II}$ consists of the blow-up of two imaginary points $q,\bar{q}$, followed by the contraction of the fibres of the first projection $\F_0\to \p^1$ passing through $q, \bar{q}$. Since the two points do not belong to the same fibre by any projection, we have $q=((a+\im b:1), (c+\im d:1))$, for some $a,b,c,d\in \r$, $bd\not=0$. There exists thus an element  $\alpha\in \Aut_\r(\F_0)$ that sends $q$ onto $p$ and then $\bar{q}$ onto $\bar{p}$. In consequence, $\tau_0\alpha (\varphi_i)^{-1}\in \Aut_\r(\p^2)$. This yields the result.
\end{proof}

\section{Other results}\label{Other}


\subsection*{Infinite transitivity on surfaces}
The group of automorphisms of a complex algebraic variety is small: indeed, it is  finite in general. Moreover, the group of automorphisms is $3$-transitive only if the variety is $\p^1$. On the other hand, it was  proved in \cite{hm3} that for a real rational surface $X$, the group of automorphisms $\Aut(X(\r))$ acts $n$-transitively on $X(\r)$ for any $n$. The next theorem determines all real algebraic surfaces $X$ having a group of automorphisms which acts infinitely transitively on $X(\r)$. 

\begin{dfn}
Let $G$ be a topological group acting continuously on a topological space $M$.
 We say that two $n$-tuples of distinct points $(p_1,\dots,p_n)$ and $(q_1,\dots,q_n)$ are \emph{compatible} if there exists an homeomorphism $\psi \colon M \to M$ such that $\psi(p_i)=q_i$ for each $i$. 
 The action of $G$ on $M$ is then said to be \emph{infinitely transitive} if for any pair of compatible $n$-tuples of points $(p_1,\dots,p_n)$ and $(q_1,\dots,q_n)$ of $M$, there exists an element $g \in G$ such that $g(p_i)=q_i$ for each $i$. More generally, the action of $G$ is said to be infinitely transitive \emph{on each connected component} if we require the above condition only in case, for each $i$, $p_i$ and $q_i$ belong to the same connected component of $M$.  
\end{dfn}

\begin{thm}\cite{blm1}
Let $X$ be a nonsingular real projective surface. The group $\aut\bigl(X(\r)\bigr)$ is then infinitely transitive on each connected component if and only if $X$ is geometrically rational and $\#X(\r) \leq3$.
 \end{thm}
 
\subsection*{Density of automorphisms in diffeomorphisms}

In \cite{km1}, it is proved that $\aut\bigl(X(\r)\bigr)$ is dense in $\diff\bigl(X(\r)\bigr)$ for the $\mathcal{C}^\infty$-topology when $X$ is a geometrically rational surface with $\#X(\r)=1$ (or equivalently when $X$ is rational). 
In the cited paper, it is said that $\#X(\r)=2$ is probably the only other case where the density holds. 
The following collect the known results in this direction. 

\begin{thm}\cite{km1,blm1}
Let $X$ be a smooth real projective surface. 
\begin{itemize}
\item If $X$ is not a geometrically rational surface, then $\overline{\aut\bigl(X(\r)\bigr)}\ne\diff\bigl(X(\r)\bigr)$;
\item If $X$ is a geometrically rational surface, then
\begin{itemize}
\item If $\#X(\r)\geq 5$, then $\overline{\aut\bigl(X(\r)\bigr)}\ne\diff\bigl(X(\r)\bigr)$;

\item if $\#X(\r)=3,4$, then for most $X$, $\overline{\aut\bigl(X(\r)\bigr)}\ne\diff\bigl(X(\r)\bigr)$;
\item if  $\#X(\r) =1$, 
then $\overline{\aut\bigl(X(\r)\bigr)}=\diff\bigl(X(\r)\bigr)$. 
\end{itemize}
\end{itemize}
Here the closure is taken in the $\mathcal{C}^\infty$-topology.
\end{thm}

\end{document}